\documentclass[11pt,noinfoline]{article}

\RequirePackage[OT1]{fontenc}
\RequirePackage[
amsthm,amsmath
]{imsart}

\usepackage{graphicx}
\usepackage{latexsym,amsmath}
\usepackage{amsmath,amsthm,amscd}
\usepackage{amsfonts}
\usepackage[psamsfonts]{amssymb}

\usepackage{mathabx} 

\usepackage{enumerate}

\usepackage{url}

\usepackage{tocvsec2}

\usepackage{epstopdf}
\usepackage{wrapfig}
\usepackage{float}
\usepackage[small]{caption}
\usepackage{hyperref}

\startlocaldefs
\numberwithin{equation}{section}

\theoremstyle{plain} 
\newtheorem{theorem}{Theorem}[section]

\newtheorem{proposition}[theorem]{Proposition}

\theoremstyle{definition} 

\theoremstyle{definition} 

\newtheorem*{ex*}{Example}
\theoremstyle{remark} 

\theoremstyle{remark} 

\newtheorem*{remark*}{Remark}

\numberwithin{equation}{section}

\makeatletter
\def\subsubsubsection{\@startsection{subsubsubsection}{4}{\z@}{-3.25ex plus -1ex minus -.2ex}{1.5ex plus .2ex}{\normalsize}}
\makeatother                                   
 
\newcommand{\beqa}{\begin{eqnarray}}
\newcommand{\eeqa}{\end{eqnarray}}
 
\newcommand{\bseq}{\begin{subequations}}
 
\newcommand{\eseq}{\end{subequations}}

\newcommand{\dd}{\partial}

\renewcommand{\dd}{{\,\operatorname{d}}}

\newcommand{\sign}{\operatorname{sign}}
\newcommand{\supp}{\operatorname{supp}}
\newcommand{\card}{\operatorname{card}}

\newcommand{\pv}{\operatorname{p{.}v{.}}}

\newcommand{\al}{\alpha}

\newcommand{\Ga}{\Gamma}
\newcommand{\si}{\sigma}

\newcommand{\ka}{\kappa}
\newcommand{\la}{\lambda}
\newcommand{\ga}{\gamma}
\newcommand{\de}{\delta}
\newcommand{\be}{\beta}
\newcommand{\De}{\Delta}
\newcommand{\La}{\Lambda}

\renewcommand{\th}{\theta}
\renewcommand{\Psi}{\overline{\Phi}}

\newcommand{\ffrown}{\text{\raisebox{3pt}[0pt][0pt]{$\frown$}}}

\renewcommand{\O}{\underset{\ffrown}{<}}

\renewcommand{\u}{\mathsf{u}}
\renewcommand{\nu}{\mathsf{nu}}

\newcommand{\ii}{\operatorname{I}}

\renewcommand{\P}{\operatorname{\mathsf{P}}} 
\newcommand{\PP}{\operatorname{\mathsf{P}}} 
\newcommand{\E}{\operatorname{\mathsf{E}}}

\newcommand{\lc}{\mathsf{L\!C}}

\newcommand{\R}{\mathbb{R}}
\newcommand{\N}{\mathbb{N}}
\newcommand{\C}{\mathbb{C}}
\newcommand{\CC}{\mathbb{C}}

\newcommand{\EE}{\mathcal{E}}

\newcommand{\GG}{\mathfrak{G}}

\newcommand{\vp}{\varepsilon}

\newcommand{\tp}{{\tilde{p}}}

\newcommand{\tPi}{{\tilde{\Pi}}}

\newcommand{\tS}{{\tilde{S}}}

\newcommand{\tX}{{\tilde{X}}}

\newcommand{\tF}{{\tilde{F}}}

\renewcommand{\le}{\leqslant}
\renewcommand{\ge}{\geqslant}

\newcommand{\ol}{\overline}

\renewcommand{\cdot}{\#}

\newcommand{\fn}{f_{1n}}
\newcommand{\fnj}[1]{f_{1n}^{(#1)}}
\newcommand{\gn}{g_{1n}}
\newcommand{\gnj}[1]{g_{1n}^{(#1)}}
\newcommand{\hn}{h_{1n}}
\newcommand{\hnj}[1]{h_{1n}^{[#1]}}

\newcommand{\dnj}[1]{d_{1n}^{(#1)}}

\renewcommand{\Re}{\operatorname{\mathrm{Re}}}
\renewcommand{\Im}{\operatorname{\mathrm{Im}}}

\endlocaldefs

\begin{document}

\begin{frontmatter}

\title{On the nonuniform Berry--Esseen bound}
\runtitle{Nonuniform Berry--Esseen}

%

\begin{aug}
\author{\fnms{Iosif} \snm{Pinelis}\thanksref{t2}\ead[label=e1]{ipinelis@mtu.edu}}
  \thankstext{t2}{Supported by NSA grant H98230-12-1-0237}
\runauthor{Iosif Pinelis}


\address{Department of Mathematical Sciences\\
Michigan Technological University\\
Houghton, Michigan 49931, USA\\
E-mail: \printead[ipinelis@mtu.edu]{e1}}
\end{aug}

\begin{abstract}
Due to the effort of a number of authors, the value $c_\u$ of the absolute constant factor in the uniform Berry--Esseen (BE) bound for sums of independent random variables 
has been gradually reduced to 
$0.4748$ in the iid case and $0.5600$ in the general case; both these values were recently obtained by Shevtsova. On the other hand, Esseen had shown that $c_\u$ cannot be less than $0.4097$. 
Thus, the gap factor between the best known upper and lower bounds on (the least possible value of) $c_\u$ is now rather close to 1. 

The situation is quite different for the absolute constant factor $c_\nu$ in the corresponding nonuniform BE bound. 
Namely, the best correctly established upper bound on $c_\nu$ in the iid case is 
over $25$ times the corresponding best known lower bound, and this gap factor is greater than 
$31$ in the general case.  
%
In the present paper, improvements to the prevailing method (going back to S.\ Nagaev) of obtaining nonuniform BE bounds are suggested. 
Moreover, a new method is presented, of a rather purely Fourier kind, based on a family of smoothing inequalities, which work better in the tail zones. As an illustration, a quick proof of Nagaev's nonuniform BE bound is given. Some further refinements in the application of the method are shown as well. 
\end{abstract}

  
%

\setattribute{keyword}{AMS}{AMS 2010 subject classifications:}




\end{frontmatter}

\settocdepth{chapter}

\tableofcontents 

\settocdepth{subsubsection}

\theoremstyle{plain} 
\numberwithin{equation}{section}


\section{Uniform and nonuniform Berry--Esseen (BE) bounds}\label{intro} 

Suppose that $X_1,\dots,X_n$ are independent zero-mean 
r.v.'s, with 
\begin{equation*}
	S:=X_1+\dots+X_n,\ A:=\sum\E|X_i|^3<\infty,\ \text{and}\  B:=\sqrt{\sum\E|X_i|^2}>0. 
\end{equation*}
Consider  
\begin{equation*}
	\De(z):=\textstyle{|\P(S>Bz)-\P(Z>z)|}\quad\text{and}\quad r_L:=A/B^3, 
\end{equation*}
where $Z\sim N(0,1)$ and $z\ge0$; of course, $r_L$ is the so-called Lyapunov ratio. 
Note that, in the ``iid'' case (when the $X_i$'s are iid), $r_L$ will be on the order of $1/\sqrt n$. 

In such an iid case, let us also assume that 
$\E X_1^2=1$. 

Uniform and nonuniform BE bounds are upper bounds on $\De(z)$ of the forms  
\begin{equation}\label{eq:BE bounds}
	c_\u\,r_L\quad\text{and}\quad c_\nu\,\frac{r_L}{1+z^3}, 
\end{equation}
respectively, 
for some absolute positive real constants $c_\u$ and $c_\nu$ and for all $z\ge0$. 

Apparently the best currently known upper bound on $c_\u$ (in the iid case) is due to Shevtsova \cite{shev11} and is given by the inequality 
\begin{equation}\label{eq:shev}
	c_\u\le0.4748. 
\end{equation}
%
On the other hand, Esseen's example \cite{esseen56} with iid $X_i$'s, $n\to\infty$, $z$ appropriately close to $0$, and 
\begin{equation}\label{eq:esseen56}
	\P(X_1=1-p_{\mathsf{Ess}})=p_{\mathsf{Ess}}=1-\P(X_1=-p_{\mathsf{Ess}})
\end{equation}
with $p_{\mathsf{Ess}}:=2-\sqrt{10}/2=0.4188...$ 
showed that $c_\u$ cannot be less than $\frac{3+\sqrt{10}}{6 \sqrt{2 \pi }}=0.4097\ldots$; a similar lower bound on the BE constant for intervals was recently shown by Dinev and Mattner \cite{din-matt} to be $\sqrt{\frac2\pi}=0.7978\ldots$, which is almost twice as large as $0.4097\ldots$. 
\label{c_u gap}. 

Thus, the optimal value of $c_\u$ is already known to be within the rather small interval from $0.4097$ to $0.4748$ in the iid case \big(in the general case the best known upper bound on $c_\u$ appears to be $0.5600$, due to 
Shevtsova \cite{shevtsova-DAN};
a slightly worse upper bound, $0.5606$, is due to Tyurin \cite{tyurin}\big).

\section{The Bohman--Prawitz--Vaaler smoothing inequalities}\label{praw}
To a significant extent the mentioned best known uniform BE bounds are based 
on the smoothing result due to Prawitz \cite[(1a, 1b)]{prawitz72_limits}, which states the following. There exists a nonempty class of functions $M\colon\R\to\CC$ such that 
\begin{equation}\label{eq:M=0}
	M(t)=0\quad\text{if}\quad|t|>1 
\end{equation}
and 
for any r.v.\ $X$, any real $T>0$, and any real $x$, 
\begin{gather}
	\GG\big(M_T(-\cdot)\E e^{iX\cdot}\big)(x)
	\le\P(X<x)-\tfrac12\le\P(X\le x)-\tfrac12
	\le\GG\big(M_T(\cdot)\E e^{iX\cdot}\big)(x), 
	\label{eq:praw}\\ 
\intertext{where}\quad 
	M_T(\cdot):=M(\cdot/T), \label{eq:M_T} \\ 
\GG(f)(x):=\frac i{2\pi}\,\pv\int_\infty^\infty e^{-itx}f(t)\frac{\dd t}t,  \label{eq:GG}
\end{gather}
and $\pv$ stands for ``principal value'', so that $\pv\int_{-\infty}^\infty:=\lim_{\vp\downarrow0\atop A\uparrow\infty}\big(\int_{-A}^{-\vp}+\int_\vp^A\big)$; here and subsequently, the symbol $\cdot$ stands for the argument of a function. 
Of course, the upper and lower bounds in \eqref{eq:praw} must take on only real values; this can be provided by the condition that 
\begin{equation}\label{eq:M1,M2}
	M_1:=\Re M\ \text{is even}\quad\text{and}\quad M_2:=\Im M\ \text{is odd.}
\end{equation}  
Note also that the upper and lower bounds in \eqref{eq:praw} easily follow from each other, by changing $X$ to $-X$. 

Inequalities \eqref{eq:praw} may be compared with the corresponding well-known inversion formula 
\begin{equation}\label{eq:ident0}
	\P(X<x)+\tfrac12\P(X=x)-\tfrac12
	=\GG(\E e^{i X\cdot})(x)
\end{equation}
for all real $x$; see e.g.\ \cite[(2)]{gurland48}.  
The multiplier $M(\cdot)$ of the c.f.\  $\E e^{iX\cdot}$ in \eqref{eq:praw} is 
the Fourier transform of the function  
$
	\check M(\cdot):=\frac1{2\pi}\,\int_{-\infty}^\infty e^{-it\cdot}M(t)\dd t,  
$ 
which may be considered as a smoothing kernel -- since, in view of \eqref{eq:M=0}, 
the spectral decomposition of $\check M$ does not have 
components of frequencies greater than $1$. 
So, the factors $M(\pm\cdot/T)$ in the bounds in \eqref{eq:praw} \emph{filter} out the components of the function $\GG(\E e^{iX\cdot})$ of frequencies greater than $T$ and thus make the function smoother and flatter, 
especially if $T$ is not large enough. 
Another way to look at such smoothing is through the Paley--Wiener theory, which implies that the Fourier spectrum of a function is contained in the interval $[-T,T]$ iff the function is (the restriction to $\R$ of) an entire analytic function of exponential type $T$ and hence rather slowly varying if $T$ is not large; see e.g.\ \cite[Section~43]{donoghue}. 
On the other hand, from an analytical viewpoint, the presence of the factors $M(\pm\cdot/T)$ is useful, because one then needs to bound the values $\E e^{itX}$ of the c.f.\  of $X$ only for $t\in[-T,T]$, which is a much easier task unless $T$ is too large.  
 
One particular function $M$ for which \eqref{eq:praw} holds is given by the formula 
\begin{equation}\label{eq:M special}
	M(t)=\big[(1-|t|)\,\pi t\cot\pi t+|t|
	-i(1-|t|)\,\pi t\big]\,\ii\{|t|<1\} 
\end{equation}
for all $t\ne0$ \cite{prawitz72_limits}; here and subsequently, it is tacitly assumed that the functions of interest are extended to $0$ by continuity.  
For this particular multiplier $M$, 
which was shown in \cite{prawitz72_limits} to have a certain optimality property, 
the corresponding smoothing kernel $\check M(\cdot):=\frac1{2\pi}\int_\R e^{-it\cdot}M(t)\dd t$ is given by the formula 
\begin{equation*}
	\check M(x)=\frac{
   2 \pi  x \sin x \left(2 \pi  (x+2 \pi )-x^2 \psi'\left(\frac{x}{2 \pi }\right)\right)
   -(1-\cos x) \left(x^3 \psi''\left(\frac{x}{2 \pi }\right)+4 \pi ^2 (x+4
   \pi )\right)}{4 \pi ^3 x^3} 
\end{equation*}

\begin{minipage}
{0.2\linewidth}
\centering
\vspace{-6pt}
\includegraphics[width=1\textwidth]{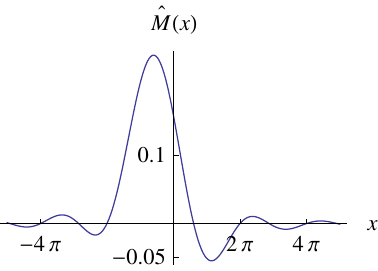}
\end{minipage}
\vspace{4pt}
\hspace{0.1cm}
\begin{minipage}
{0.73\linewidth}
for 
$x\notin\{-2n\pi\colon n\in\{0\}\cup\N\}$, 
where $\psi$ is the digamma function, defined by the formula $\psi(z)=\Ga'(z)/\Ga(z)$;   
this kernel is (necessarily
) asymmetric and alternating in sign; also, $\int_{-\infty}^\infty\check M(x)\dd x=M(0)=1$; a part of the graph of this kernel $\check M$ is shown here on the left. 
\end{minipage}

Earlier, inequalities of the form \eqref{eq:praw} were obtained by Bohman~\cite{bohman} for another class of functions $M$, with apparently not quite as good approximation properties. 
Another approach to Prawitz's results was demonstrated by Vaaler~\cite{vaaler85}. 


\section{Nonuniform BE bounds: Nagaev's result and method}\label{nonunif} 
The classical result by Nagaev \cite{nagaev65} is that in the ``iid'' case 
\begin{equation}\label{eq:BE nonunif}
	|\P(S>z\sqrt n)-\P(Z>z)|\le c_\nu\frac{\E|X_1|^3}{(1+z^3)\sqrt n}
\end{equation}
for all real $z\ge0$, where $c_\nu$ is an absolute constant. 
Bikelis \cite{bik66} extended this result to the case of non-iid $X_i$'s. 
Nagaev's method involves the following essential components:
\begin{itemize}\label{nag-method}
	\item truncation;
	\item Cramer's exponential tilt, together with a uniform BE bound;  
	\item an exponential bound on large deviation probabilities.
\end{itemize}
First, truncated versions of $X_i$, say $X_i^{(y)}$, are obtained, such that $X_i^{(y)}\le y$ for some real $y>0$ and all $i$ \big(the r.v.'s $X_i^{(y)}$ may, in some variants of this approach including \cite{nagaev65}, be improper in the sense that they may take values that are not real numbers\big). The truncation is done in order to make the exponential tilt and an exponential inequality possible. 
The value of the truncation level $y$ is chosen (i) to be large enough so that the tails of the truncated sum $S^{(y)}:=X_1^{(y)}+\dots+X_n^{(y)}$ be close enough to those of $S$ and, on the other hand, (ii) to be small enough so that the exponential tilt and the exponential inequality result in not too large a bound. In some variants, including the ones in \cite{nagaev65,bik66}, two different truncation levels are used. 

In view of the uniform BE bound
, without loss of generality $z\ge z_0$, where $z_0$ is an arbitrarily chosen positive real number. Two main cases are then considered:  
\begin{enumerate}[{Case} 1: ]
	\item $z_0\le z<c\sqrt{\ln(\sqrt n/\E|X_1|^3)}$ (``moderate deviations''); 
	\item $z\ge z_0\vee c\sqrt{\ln(\sqrt n/\E|X_1|^3)}$ (``large deviations''); 
\end{enumerate}
here $c$ is a positive constant. 


In Case~1, of moderate deviations, the exponential tilting is performed, which may be presented as follows. Take some real $h>0$ and let 
$\tX_1=\tX_1^{(h,y)},\dots,\tX_n=\tX_n^{(h,y)}$ be any r.v.'s such that 
\begin{equation}\label{eq:tilt}
	\E g(\tX_1,\dots,\tX_n)=\frac{\E e^{hS^{(y)}}g(X_1^{(y)},\dots,X_n^{(y)})}{\E e^{hS^{(y)}}}
\end{equation}
for all bounded (or for all nonnegative) Borel-measurable functions $g\colon\R^n\to\R$.  
Equivalently, one may require condition \eqref{eq:tilt} only for Borel-measurable indicator functions $g$; clearly, such r.v.'s $\tX_i$ do exist. It is also clear that the r.v.'s $\tX_i$ are independent. 
These r.v.'s, the $\tX_i$'s, may be referred to as the tilted or, more specifically, $h$-tilted versions of the $X_i^{(y)}$'s. 
Clearly, without the truncation, the tilted versions of the original r.v.'s $X_i$ may not exist, since $\E e^{hS}$ may be infinite even if $\E|X_i|^3<\infty$ for all $i$.   
Using \eqref{eq:tilt} with $g(x_1,\dots,x_n)=e^{-h(x_1+\dots+x_n)}\ii\{x_1+\dots+x_n>x\}$, it is easy to see that 
\begin{equation}\label{eq:tilt back}
	\P(S^{(y)}>x)=\E e^{hS^{(y)}}\,\int_x^\infty \dd u\; h e^{-hu}\P(x<\tS\le u)
\end{equation}
for all real $x$, where 
$
	\tS:=\tX_1+\dots+\tX_n. 
$ 
Similarly, one can write 
\begin{equation}\label{eq:tilt normal}
	\P(BZ>x)=\E e^{hBZ}\,\int_x^\infty \dd u\; h e^{-hu} \P(x<BZ+B^2h\le u)
\end{equation}
for all real $x$, since any $h$-tilted version of the r.v.\ $BZ$ has the distribution $N(B^2h,B^2)$. 
At that, good choices for $y$ and $h$ are of the form $\al x$ and $\eta x/B^2$, for some real parameters $\al$ and $\eta$ in $(0,1)$. 

So, to bound $|\P(S^{(y)}>z\sqrt n)-\P(Z>z)|$ (cf.\ \eqref{eq:BE nonunif}), one can demonstrate sufficient closeness of the terms $\E e^{hS}$ and $\P(x<\tS\le u)$ in \eqref{eq:tilt back} to the corresponding terms $\E e^{hBZ}$ and $\P(x<BZ+B^2h\le u)$ in \eqref{eq:tilt normal}. 
For each $i$, one notices that 
\begin{equation}\label{eq:tilted 3rd}
	\E|X_i^{(y)}|^3e^{hX_i^{(y)}}\le e^{hy}\E|X_i|^3  
\end{equation}
and then shows that $\E e^{hX_i^{(y)}}$ is close 
enough to $1$ and, somewhat more precisely, to $\E e^{hZ\sqrt{\E X_i^2}}$, and that 	
the mean and variance of $\tX_i$ are close enough to $h\E X_i^2$ and
$\E X_i^2$, respectively. So, one shows that $\E e^{hS^{(y)}}$ is close to $\E e^{hBZ}$, and the first two moments of $\tS^{(y)}$ are close enough to those of $BZ+B^2h$. 
Using now a uniform BE bound 
as in \eqref{eq:BE bounds} -- but for the $\tX_i$'s rather than the $X_i$'s, one shows that $\P(x<\tS\le u)$ is close enough to $\P(x<BZ+B^2h\le u)$.    

In Case~2, of large deviations, instead of the exponential tilting and a uniform BE bound, one employs an exponential inequality to bound $\P(S^{(y)}>x)$ and hence $\P(S^{(y)}>x)-\P(BZ>x)$ from above; for the lower bound on the latter difference, one simply uses $-\P(BZ>x)$. 

\subsection{A historical sketch of the problem of nonuniform BE bounds 
} \label{errors}
The constant factors $c_\nu$ 
in the mentioned papers \cite{nagaev65,bik66} were not explicit. 
All papers known to this author with explicit values of $c_\nu$ followed the scheme of proof given by Nagaev \cite{nagaev65}, as delineated above. 

Apparently the first such explicit value of $c_\nu$ was greater than $1955$, as reported by Paditz \cite{paditz78}. 
In his dissertation
\cite{paditz_diss}, a much better value, $114.7$, was presented. 
Later, Paditz \cite{paditz89} showed that $c_\nu<31.935$. 

Michel \cite{michel81} showed that in the iid case $c_\nu\le c_u+8(1+e)$, which would be less than $30.2211$, assuming the mentioned value $0.4748$ for $c_\u$, obtained in the later paper by Shevtsova \cite{shev11}. 

Again in the iid case, Nefedova and Shevtsova \cite{nefed-shev11} briefly stated that they had gone along the lines of the proof in \cite{paditz89} except using a better value for $c_\u$ (namely, $0.4784$, obtained in \cite{kor-shev10-actuar}) in place of such a value (namely, $0.7915$ \cite{shiganov}) used in \cite{paditz89}, to get $25.80$ for $c_\nu$. 

Once again in the iid case, Nefedova and Shevtsova claimed in \cite{nefed-shev} that $c_\nu<18.2$. 
However, there appears to be an error there. 
Namely, the first inequality in \cite[(14)]{nefed-shev} is equivalent to the reverse of the last inequality on page 75 there, which latter is in turn equivalent to the condition $x^2\ge c_n(x;\de,a,b,c)$ in \cite[Theorem 1]{nefed-shev}, which is also equivalent to the second display on \cite[page~75]{nefed-shev}; the expression $c_n(x;\de,a,b,c)$ is defined in the first display on page~70 of \cite{nefed-shev}. 
So, for any given $X,n,b,c,\de$ satisfying all the conditions of \cite[Theorem 1]{nefed-shev}, the first inequality in \cite[(14)]{nefed-shev} and the last inequality on \cite[page~75]{nefed-shev} can both hold only for one value of $a$. This wrong inequality in \cite[(14)]{nefed-shev} is also used for \cite[(16)]{nefed-shev}. 

Finally, working along lines quite similar to those in \cite{nefed-shev}, Grigor'eva and Popov \cite{grig-popov2012Dokl,grig-popov2012SSI} claimed that $c_\nu<22.2417$ in the general, non-iid case. However, there appears to be the same kind of errors there: compare \cite[(9) and (11)]{grig-popov2012SSI} with \cite[(14) and (16)]{nefed-shev}, respectively.

This leaves, for now, 
$31.935$ as the best (possibly correctly) established nonuniform BE constant factor $c_\nu$ -- in the general, non-iid case. 
On the other hand, it follows from a result by Chistyakov \cite[Corollary~1]{chist90} that $c_\nu$ is necessarily no less than $1$, and this lower bound on $c_\nu$ is asymptotically exact in a certain sense for $z\to\infty$. Apparently, this has been the best known lower bound on $c_\nu$. However, it is easy to improve this bound slightly and show that necessarily 
\begin{equation}\label{eq:c_nu>}
	c_\nu>1.0135; 
\end{equation}
this can be done by letting $X_1$ have the centered Bernoulli distribution with parameter $p=8/100$ and then letting $n=1$ and $z\uparrow1-p$. 
However, it was shown by Bentkus \cite{bent94BE}, the best constant factor for $n=1$ will be $1$ if $1+z^3$ in \eqref{eq:BE nonunif} is replaced by $z^3$; it is also conjectured in \cite{bent94BE} that the same constant factor, $1$, will be good for all $n$. 

Thus, in the non-iid case the apparently best known lower bound on $c_\nu$ is over 
$31$ times smaller than the best established upper bound on $c_\nu$, and this gap factor is 
over $25$ 
in the iid case. 

\subsection{Possible improvements of Nagaev's method}\label{nag-improve}
A crucial component of the mentioned method offered by Nagaev~\cite{nagaev65} and used in the subsequent papers \cite{bik66,paditz78,paditz_diss,michel81,tysiak,mirakh84,padit-mirakh86,paditz89,gavr,nefed-shev11,nefed-shev} 
is an exponential inequality. 
However, the exponential bounds used in all of those papers are not the best possible ones. An optimal exponential bound, in terms of the first two moments and truncated absolute third moments of the $X_i$'s was given by Pinelis and Utev \cite{pin-utev-exp}. In fact, the paper \cite{pin-utev-exp} provided a general method to obtain optimal exponential bounds, along with a number of specific applications of the general method. 

However, even the best possible \emph{exponential} bounds, say for sums of independent r.v.'s, can be significantly improved. 
The reason for this is that the class of exponential moments functions is very small (even though analytically very simple to deal with). Using a much richer class of moments functions, Pinelis \cite{pin-hoeff-AIHP} obtained the following result. %
Let $X_1,\dots,X_n$ be independent random variables (r.v.'s), with the sum $S:=X_1+\dots+X_n$. 
For any $a>0$ and $\th>0$, let  
$\Ga_{a^2}$ and $\Pi_{\th}$ stand for any independent r.v.'s such that  
$\Ga_{a^2}$ has the normal distribution with parameters $0$ and $a^2$, and $\Pi_{\th}$
has the Poisson distribution with parameter $\th$. 
Let also 
$
	\tPi_{\th}:=\Pi_{\th}-\E\Pi_{\th}=\Pi_{\th}-\th.
$ 
Let $\si$, $y$, and $\be$ be any 
positive real numbers  such that 
$
	\vp:=\frac{\be}{\si^2y}\in(0,1).     
$ 
Suppose that  
$
	\sum_i\E X_i^2\le\si^2,\ \sum_i\E(X_i)_+^3\le\be,\ \E X_i\le0,\ \text{and $X_i\le y$, 
	}
$ 
for all $i$.  
Let 
$
	\eta_{\vp,\si,y}:=\Ga_{(1-\vp)\si^2}+y\tPi_{\vp\si^2/y^2}. 
$ 
Then it is proved in \cite{pin-hoeff-AIHP} that 
\begin{equation}\label{eq:PUfF3}
	\E f(S)\le\E f(\eta_{\vp,\si,y})
\end{equation}
for all twice continuously differentiable functions $f$ such that $f$ and $f''$ are nondecreasing and convex. 
A corollary of this result is that for all $x\in\R$
\begin{align}
	\PP(S\ge x)\le
	\inf_{t\in(-\infty,x)}\,\frac{\E(\eta_{\vp,\si,y}-t)_+^3}{(x-t)^3} 
&\le c_{3,0}\,\PP^\lc(\eta_{\vp,\si,y}\ge x),\label{eq:main LC}
\end{align}
where 
$c_{3,0}:=\frac{2e^3}9\approx4.46$ 
and 
the function $\R\ni x\mapsto\PP^\lc(\eta\ge x)$ is defined as the least log-concave majorant over $\R$ of the tail function $\R\ni x\mapsto\PP(\eta\ge x)$ of a r.v.\ $\eta$. 
The bounds in \eqref{eq:PUfF3} and \eqref{eq:main LC} are much better than even the best exponential bounds (expressed in the same terms). 

A trade-off here is that the 
bounds 
given 
in \eqref{eq:main LC} are significantly more difficult to deal with, especially analytically, than exponential bounds. However, 
this can be done, as shown in the following 
discussion. 
In accordance with what was pointed out above, one needs an exponential bound (or a better one) only in Case~2, of large deviations, when $z\ge z_0\vee c\sqrt{\ln(\sqrt n/\E|X_1|^3)}$, which implies 
\begin{equation}\label{eq:case2}
	\E|X_1|^3/\sqrt n\ge e^{-z^2/c^2}. 
\end{equation}
Also, 
by \eqref{eq:main LC}, for any real $\tau<1$
\begin{equation*}
	(1-\tau)^3\P(S^{(y)}>x)\le\sum_{j=0}^\infty Q_j\frac{\la^j}{j!}e^{-\la},
\end{equation*}
where 
\begin{gather}
	x:=Bz=z\sqrt n,\quad y=\al x,\quad \al\in(0,1), \quad Q_j:=\Big(\frac{\al_1}z\Big)^3\E(Z+u_j)_+^3,
	\notag\\
	\al_1:=\sqrt{1-az_0^2/\al},\quad a:=\frac{\E(y\wedge (X_1)_+)^3}{z^3\sqrt n}, \quad u_j:=\al\Big(j-\frac\tau\al-\la\Big)z, \quad 
	\la:=\frac a{\al^3}. \label{eq:la} 
\end{gather}
Assume now that 
$
	\tau c^2\ge2, 
$ 
where $c$ is as in 
\eqref{eq:case2}. 
Since $z\ge z_0$, one has 
\begin{equation}\label{eq:Q0 bound}
	Q_0\le C_0\frac{\E|X_1|^3}{z^3\sqrt n}, 
\end{equation}
where 
\begin{equation}\label{eq:C0}
	C_0:=e^{z_0^2/c^2}\E(Z-\tau z_0)_+^3; 
\end{equation}
here one uses the fact that $e^{\be t^2}\E(Z-t)_+^3$ is decreasing in $t\ge0$ provided that $\be\le1/2$; in fact, this decrease is fast, especially when $\be<1/2$. 
Note also that $\E(Z-t)_+^3=\left(t^2+2\right) \varphi (t)-t \left(t^2+3\right) \bar{\Phi }(t)$ for all real $t$, where $\varphi$ and $\bar{\Phi }$ are the density and tail functions of $Z$. 

Next, since $\E g(\be Z)$ is nondecreasing in $\be\ge0$ for any convex function $g$, 
\begin{equation*}
	Q_j\le(\al_1/z_0)^3\E(Z+u_{j0})_+^3, \quad\text{where}\quad u_{j0}:=\al(j-\tau/\al-\la)z_0. 
\end{equation*}
Using now the identity $\E(Z+t)_+^3=t^3+3t+\E(Z-t)_+^3$ for all real $t$ and the decrease of $\E(Z-t)_+^3$ in $t\in\R$, one has 
\begin{equation}\label{eq:Q>0 bound}
	\sum_1^\infty Q_j\frac{\la^j}{j!}e^{-\la}
	\le C_1(a)\frac{\E|X_1|^3}{z^3\sqrt n}, 
\end{equation}
where 
\begin{equation*}
	C_1(a):=\Big(\frac{\al_1}{\al z_0}\Big)^3\Big[\sum_1^\infty(u_{j0}^3+3u_{j0})\frac{\la^{j-1}}{j!}
	+\E(Z-u_{10})_+^3+\E(Z-u_{20})_+^3\frac{\la^{2-1}}{2!}e^\la\Big]e^{-\la}, 
\end{equation*}
with $\la$ as defined in \eqref{eq:la}. 
The sum $\sum_1^\infty$ in the above expression of $C_1(a)$ is easy to evaluate explicitly. 
Also, since the left-hand side of \eqref{eq:BE nonunif} can never exceed $1$, without loss of generality 
\begin{equation}\label{eq:amax}
	a\le a_{\max},
\end{equation}
with $a_{\max}=1/c_\nu$; 
working a bit harder, one may assume \eqref{eq:amax} with $a_{\max}$ significantly smaller than $1/c_\nu$. 
Next, it appears that for values of the parameters $\al$ and $\tau$ that have a chance to be optimal or quasi-optimal, the factor $C_1(a)$ will be decreasing in $a\in[0,a_{\max}]$. 
Therefore and in view of \eqref{eq:Q0 bound} and \eqref{eq:Q>0 bound}, one will have 
\begin{equation}\label{eq:LD}
	\P(S^{(y)}>x)\le\frac{C_0+C_1(0+)}{(1-\tau)^3}\,\frac{\E|X_1|^3}{z^3\sqrt n}, 
\end{equation}
with $C_0$ as in \eqref{eq:C0} and $C_1(0+)=(\al z_0)^{-3}\E(Z+(\al-\tau)z_0)_+^3$. 
One can improve the above estimates by partitioning the interval $[0,a_{\max}]$ into a number of smaller subintervals and then considering the corresponding cases depending on which of the subintervals the value of $a$ is in. 

Thus, it is shown that $\P(S^{(y)}>x)$ can be appropriately bounded using the better-than-exponential bound in \eqref{eq:main LC}, and at that in a rather natural manner and incurring almost no losses. 
It should be clear that the expression on the right-hand side of inequality \eqref{eq:LD} will become a term in a bigger expression that is an upper bound on the left-hand side of \eqref{eq:BE nonunif}. That latter, bigger expression will then have to be (quasi-)minimized with respect to $z_0$, $\al$, $\eta$, $\tau$, and the other parameters, subject to the necessary restrictions on their values. 

Also, one can use ideas from \cite{a.nagaev69-I,a.nagaev69-II,pin81,pin85} to improve the estimation of the effect of truncation, as compared with the way that was done in the mentioned papers \cite{nagaev65,bik66,paditz78,paditz_diss,michel81,tysiak,mirakh84,padit-mirakh86,paditz89,gavr,nefed-shev11,nefed-shev}, as well as more ``synthetic'' ways to bound moments of the tilted distribution -- cf.\ results in \cite{winzor,pin11tilt,tilt-symm,re-center}.  
In addition, as in \cite{nefed-shev}, one can use the uniform bound $0.3328(\E|X_1|^3+ 0.429)/\sqrt n$ from \cite{shev11}, which is smaller than the previously mentioned bound of the classical form $c_\u L=c_u\E|X_1|^3/\sqrt n$ with $c_\u= 0.4748$. 
There are a few other potentially useful modifications. 
Thus, the improvements concern every one of the three major ingredients of Nagaev's method listed beginning on page~\pageref{nag-method}. 
By utilizing the above ideas, one may hope to 
improve the upper bound on $c_\nu$ to about 10 in the iid case and to about 12 in the general case. 
When and if such an objective is attained, the gap between the available upper and lower bounds on $c_\nu$ will be decreased, at least, about $3$ times in the iid case and about $10$ times in the general case. 

However, significant further progress after that seems unlikely within the framework of the method of \cite{nagaev65}. 
One of the main obstacles here is the factor $e^{hy}$ as in \eqref{eq:tilted 3rd}. Since good choices for $y$ and $h$ turn out to be $\al x$ and $\eta x/B^2$ with $\al$ and $\eta$ somewhat close to $0.5$, this factor will then be something like $e^{z^2/4}$, which is large for large enough $z$. 

Yet, the factor $e^{hy}$ is the best possible one in \eqref{eq:tilted 3rd} (even assuming that $X^{(y)}_i=X_i$ and hence $X^{(y)}_i$ is zero-mean). Such a large factor is necessary when $X_i$ has a two-point distribution highly skewed to the right. 
On the other hand, certain considerations suggest that the least favorable situation in Case~1 of moderate deviations is when $n$ is very large but $z$ is not so, and then the mentioned 
least favorable distribution (for the uniform BE bound) given by \eqref{eq:esseen56} is only slightly skewed. 
This creates a significant tension in using the exponential tilt. 

One may try to reduce the factor $e^{hy}$ by decreasing $\al$ and hence $y$ -- but this will increase fast the effect of truncation, which is (at least roughly) proportional to $1/\al^3$. 

Even if one were able to get rid of the factor $e^{hy}$ altogether, the corresponding uniform BE bound on the rate of convergence to the probability $\P(x<BZ+B^2h\le u)$ in \eqref{eq:tilt normal} would still seem relatively too large, since this probability itself is less than $\P(x<BZ+B^2h)=\P\big(Z>(1-\eta)z\big)$ and therefore is rather small for 
what appears to be the least favorable 
values of $z$
, such as $2.5$ to $3.5$ (and values of $\eta$ typically not too far from $0.5$). In contrast, the mentioned asymptotic lower bound by Esseen~\cite{esseen56} (recall \eqref{eq:esseen56}) is attained for $z$ close to $0$; furthermore, the corresponding asymptotic expression is rather highly peaked near the maximum and is thus much smaller outside of a neighborhood of the maximum point.   

Yet another apparently powerful cause of tension is as follows. After the $X_i$'s have been truncated, a natural bound on $|\P(S^{(y)}>x)-\P(BZ>x)|$, obtained via either the exponential tilt or a Stein-type method, decays in an exponential rather than power fashion; see e.g.\ the results \cite{chen-shao05} and \cite{BE-lin}, which imply an upper bound of the form $c(\la)\frac{\E|X_1|^3}{e^{\la z}\sqrt n}$ for real $\la>0$, say in the iid case. 
The factor $1/e^{\la z}$ decays much faster than $1/(1+z^3)$ when $z$ is large. However, the former factor may be much greater than the latter, especially if $\la$ is not large and $z$ is not very large. For instance, if $\la=1/2$ as in \cite{chen-shao05}, then 
$\max_{z>0}\frac1{e^{\la z}}\big/ \frac1{1+z^3}=10.8\dots$, attained at $z=5.9719\dots$.

In the next section, a new approach to obtaining nonuniform BE bounds
is described, based on the Fourier method, complemented by extremal problem methods. 

\section{A 
new way to obtain nonuniform BE bounds} \label{new way}

Take any function 
$h\in C^1$ such that (the limit) $\GG(h)$ exists (and is) in $\R$ and $h(t)/t\to0$ as $|t|\to\infty$; here and in what follows, $C^k$ denotes the class of all $k$ times continuously differentiable complex-valued functions defined on $\R$. 
Take any real $x\in\R$. 
Note that $\GG(1)(x)=\frac12\,\sign x$ -- say, by \eqref{eq:ident0} with $X=0$. So, writing  $\GG(h)=h(0)\GG(1)+\GG\big(h-h(0)\big)$ and evaluating the 
$\pv$-integral 
in the expression for $\GG\big(h-h(0)\big)$ by parts, one has 
\begin{equation}\label{eq:GG1} 
	x\GG(h)(x)=\tfrac12\,h(0)\,x\sign x + i\,\GG(\La h)(x)  
\end{equation}
if $x\ne0$, 
where the linear operator $\La$ is defined by the formula 
\begin{equation}\label{eq:1st round}
	(\La h)(t):=-t\,\frac{\dd}{\dd t}\,\frac{h(t)-h(0)}t=-\frac{h(0)-[h(t)+h'(t)(-t)]}t
\end{equation}
for $t\ne0$. 
In fact, identity \eqref{eq:GG1} holds for $x=0$ as well, in view of 
the definitions of $\GG$ and $\La$. 
By induction, for all $k\in\N:=\{1,2,\dots\}$ 
\begin{equation}\label{eq:La_k}
	(\La^k h)(t)=-k!\,t^{-k}\,\Big(h(0)-\sum_{j=0}^k h^{(j)}(t)\,\frac{(-t)^j}{j!}\Big) 
	=(-1)^k\int_0^1 [h^{(k)}(t)-h^{(k)}(\al t)]\,k\al^{k-1}\dd\al 
\end{equation}
if $h\in C^k$ and $t\ne0$, and hence $(\La^k h)(0)=0$. 
So, iterating \eqref{eq:GG1}, one has 
\begin{equation}\label{eq:GGk} 
	\GG(h)(x)=\frac{h(0)\sign x}2  + \Big(\frac ix\Big)^k\,\GG(\La^k h)(x)  
\end{equation}
for all real $x\ne0$, all $k\in\N$, and functions $h$ such that 
\begin{equation}\label{eq:g}
	\text{$h\in C^k$,\quad $\GG(h)$ exists in $\R$,\quad and\quad $h^{(j)}(t)/t\to0$ for all $j=1,\dots k$ as $|t|\to\infty$.}
\end{equation} 
More generally, 
\begin{equation}\label{eq:GGk gen} 
	\GG(h)(x)=\frac{h(0)\sign x}2 
	+\frac i{2\pi}\sum_{j=1}^k\frac{h^{(j)}(0+)-h^{(j)}(0-)}{j(ix)^{j}}
	+ \Big(\frac ix\Big)^k\,\GG(\La^k h)(x)  
\end{equation}
for all real $x\ne0$, all $k\in\N$, and all functions $h$ such that 
\begin{multline}\label{eq:g gen}
	\text{$h\in C(\R)$, \quad $h\in C^k(\R\setminus\{0\})$,\quad $\GG(h)$ exists in $\R$,\quad and} \\ 
	\text{for each $j\in\{1,\dots k\}$\   
	there exists $h^{(j)}(0\pm)\in\R$\    
	 and\ $h^{(j)}(t)/t\to0$ as $|t|\to\infty$.}
\end{multline} 
The condition $h\in C^k(\R\setminus\{0\})$ in \eqref{eq:g gen} can be slightly relaxed, to the following: 
\begin{equation}\label{eq:g-var}
\begin{gathered}
	\text{$h\in C^{k-1}(\R\setminus\{0\})$\quad and\quad $h^{(k-1)}$ is of locally bounded variation on $\R\setminus\{0\}$,} 
	\end{gathered}
\end{equation}
with $\GG(\La^k h)(x)$ then understood as $\GG(\widetilde{\La^k}h)(x)
+(-1)^k\tilde\GG(h^{(k-1)})(x)$, where 
$(\widetilde{\La^k}h)(t):=\break 
-k!\,t^{-k}\,\big(h(0)-\sum_{j=0}^{k-1} h^{(j)}(t)\,\frac{(-t)^j}{j!}\big)$ for $t\ne0$ 
and 
$\tilde\GG(h^{(k-1)})(x):=
\frac i{2\pi}\pv\int_{-\infty}^\infty e^{-itx}\frac{\dd h^{(k-1)}(t)}t$. 

Identity \eqref{eq:GGk} immediately implies  
\begin{theorem}\label{th:ineq k}
Take any $k\in\N$ and any real $T>0$ and $x\ge0$. 
Let $X$ be any r.v.\ with $\E|X|^k<\infty$. Let $f$ denote the c.f.\ of $X$. 
Let $M$ be as in \eqref{eq:praw}, with the additional requirement that $M\in C^k$. 
Then 
\begin{equation}\label{eq:ineq_k}
	-i^k\GG\big(\La^k\, r_{T,-}\big)(x)\le x^k\P(X>x)
	\le x^k\P(X\ge x)\le-i^k\GG\big(\La^k\, r_{T,+}\big)(x), 
\end{equation}
where 
\begin{equation}\label{eq:r}
	r_{T,\pm}(\cdot):=M_T(\mp\cdot)f(\cdot). 
\end{equation}
\end{theorem}

\begin{remark*} Condition $\E|X|^k<\infty$ in Theorem~\ref{th:ineq k} implies $f\in C^k$, so that \eqref{eq:g} holds with $g=r_{T,\pm}$. 
\end{remark*}

As was mentioned, 
the Prawitz smoothing filter $M$ given by \eqref{eq:M special} provides the tightest, in a certain sense, upper and lower bounds in \eqref{eq:praw} on the d.f.\ of $X$. 
However, it is not smooth enough to be used in Theorem~\ref{th:ineq k} in the most interesting in applications case $k=3$. 
Namely, that $M$ is not even in $C^1$ -- whereas one needs $M\in C^3$ in Theorem~\ref{th:ineq k} for $k=3$.

There are a number of ways to develop such a smooth enough smoothing filter. 
Some of them can be based on Proposition~\ref{prop:bohman} in Section~\ref{smoothing} of this paper; see e.g.\ the function $M=M_{0,2}$ given by formula \eqref{eq:Mmy special}. 


The identity \eqref{eq:ident0} can be rewritten in the following more general and hence sometimes more convenient form. 

\begin{proposition}\label{prop:b-var}
Let $L$ be any complex-valued function of bounded variation on $\R$, and let $\ell$ be its Fourier--Stieltjes transform, so that 
$
	\ell(t)=\int_{-\infty}^\infty e^{i t x}\dd L(x) 
$ 
for all real $t$. 
Assume also that $L$ is regularized so that $2L(x)=L(x-)+L(x+)$ for all $x\in\R$ and extended to $[-\infty,\infty]$ so that $L(\pm\infty)=\lim_{x\to\pm\infty}L(x)$. 
Then 
\begin{equation}\label{eq:b-var}
	L(x)-\tfrac12\,[L(\infty)-L(-\infty)]
	=\GG(\ell)(x) \quad\text{for all real $x$.} 
\end{equation}
\end{proposition}

This follows immediately from \eqref{eq:ident0}, because (i) both sides of \eqref{eq:b-var} are linear in $L$ and (ii) any regularized function of bounded variation on $\R$ is a linear combination (with complex coefficients) of regularized distribution functions.  

Suppose that $\ell\colon\R\to\C$ is a function which may depend on a number of parameters. 
For brevity, let us say that the function $\ell$ is a \emph{quasi-c.f.} if it can be represented as a linear combination of $k$ c.f.'s with (possibly complex) coefficients such that the length $k$ of the combination and the coefficients are bounded uniformly over all possible values of the parameters. 

Clearly, the product of two quasi-c.f.'s is a quasi-c.f. 
Also, any linear combination of two quasi-c.f.'s is a quasi-c.f., provided that the coefficients of the  combination are bounded uniformly over all possible values of the parameters.  
Moreover, one has the following simple proposition. 

\begin{proposition}\label{prop:quasi-cf} 
Take any natural $m$. 
Let $N$ denote the c.f.\ of a r.v.\ $Y$ whose distribution may depend on a number of parameters. 
Suppose that $\E|Y|^m$ is (finite and) bounded uniformly over all possible values of the parameters. 
Then the $m$th derivative $N^{(m)}$ of $N$ is a quasi-c.f. 
\end{proposition}

\begin{proof}[Proof of Proposition~\ref{prop:quasi-cf}] 
Let us exclude the trivial case when $\E|Y|^m=0$. 
If $m$ is even or $\E Y_+^m=0$ or $\E Y_-^m=0$, then $\widetilde{N^{(m)}}(\cdot):=\frac{\E Y^m e^{iY\cdot}}{\E Y^m}$ is a c.f., and $N^{(m)}=(i^m\E Y^m)\widetilde{N^{(m)}}$, so that $N^{(m)}$ is a quasi-c.f. In the remaining case one has $\E Y_+^m>0$ and $\E Y_-^m>0$, so that one can similarly write 
$N^{(m)}(\cdot)=i^m\E Y_+^m\, \frac{\E Y_+^m e^{iY\cdot}}{\E Y_+^m} 
+ (-i)^m\,\E Y_-^m \frac{\E Y_-^m e^{iY\cdot}}{\E Y_-^m}$. 
\end{proof}

\paragraph{}
\label{quick-proof}
\emph{A quick proof of Nagaev's nonuniform BE bound \eqref{eq:BE nonunif}} can be easily obtained based on Theorem~\ref{th:ineq k}. 
%
Indeed, let $T=c_T\sqrt{n}/\be_3$, where $\be_3:=\E|X_1|^3$ and $c_T$ is a small enough positive real constant. Let $A\O B$ mean $|A|\le CB$ for some absolute constant $C$. 
Let $X:=S/\sqrt n$. 
If $T\le1$ then $1\O\frac{\be_3}{\sqrt n}$. So, for all real $x\ge0$, by the Markov and Rosenthal inequalities, 
$(1+x^3)\P(X\ge x)\le1+\E|X|^3\O 1+\frac{\be_3}{\sqrt n}\O\frac{\be_3}{\sqrt n}$ and similarly 
$(1+x^3)\P(Z\ge x)\O\frac{\be_3}{\sqrt n}$, whence \eqref{eq:BE nonunif} follows. 

It remains to consider the case $T>1$. 
Note that then $n>(\be_3/c_T)^2\ge3$ and hence $n\ge4$ provided that $c_T\le1/\sqrt3$. 

In view of the uniform BE bound, Theorem~\ref{th:ineq k}, and \eqref{eq:La_k}, in order to prove \eqref{eq:BE nonunif} 
it is enough to show that 
$\GG\big(r_{1,f}'''(\al\cdot)-r_{1,g}'''(\al\cdot)\big)\O\be_3/\sqrt n$ and 
$
\GG\big(r_{2,f}'''(\al\cdot)\big)\O\be_3/\sqrt n$ over $\al\in(0,1]$, where 
$r_{j,f}(\cdot):=
M_j(\frac{\cdot}T)f(\cdot)$, $j=1,2$, 
the $M_j$'s are as in \eqref{eq:M1,M2}, $M$ is (say) as in \eqref{eq:Mmy special}, $f$ is the c.f.\ of $X:=S/\sqrt n$, and $g(\cdot):=e^{-\cdot^2/2}$ (so that $g$ may be considered as a special case of $f$). 
One has  
\begin{equation}
r_{j,f}'''(\cdot)=\sum_{q=0}^3{3\choose q}\frac1{T^q}M_j^{(q)}\Big(\frac\cdot T\Big) f^{(3-q)}(\cdot). 
\end{equation}
By \eqref{eq:Mmy special} and \eqref{eq:pMy-4th}, $M_1$ is the c.f.\ of a distribution with a finite 4th moment, whereas $M_2=\ka M_1'$. Hence, by Proposition~\ref{prop:quasi-cf}, $M_j^{(q)}$ is a quasi-c.f.\ for each pair $(j,q)\in\{1,2\}\times\{0,1,2,3\}$, and then so is $M_j^{(q)}(\frac\cdot T)$. 
Similarly, $f$ is the c.f.\ of the r.v.\ $X$ with $\E|X|^3\O 1+\frac{\be_3}{\sqrt n}\O1$, by the Rosenthal inequality and the case condition $T>1$. So, again by Proposition~\ref{prop:quasi-cf}, $f^{(3-q)}$ is a quasi-c.f.\ for each $q\in\{0,1,2,3\}$. Thus, $M_j^{(q)}(\frac{\al\cdot}T)f^{(3-q)}(\al\cdot)$ is a quasi-c.f.\ and, by Proposition~\ref{prop:b-var}, $\GG\big(M_j^{(q)}(\frac{\al\cdot}T)f^{(3-q)}(\al\cdot)\big)\O1$,  
for each $(j,q)\in\{1,2\}\times\{0,1,2,3\}$. 
Therefore and because $T>1$, 
\begin{equation}
	\GG\Big({3\choose q}\frac1{T^q}M_j^{(q)}(\tfrac{\al\cdot}T)f^{(3-q)}(\al\cdot)\Big)
	\O\frac1{T^q}\le\frac1T\O\frac{\be_3}{\sqrt n} \quad
	\text{for each $(j,q)\in\{1,2\}\times\{1,2,3\}$;} 
\end{equation}
note that $q=0$ is not included here. 

It remains to show that $\GG_{1\al}(f'''-g''')\O\frac{\be_3}{\sqrt n}$ and $\GG_{2\al}(f''')\O\frac{\be_3}{\sqrt n}$ 
for $\al\in(0,1]$, 
where 
\begin{equation}\label{eq:G al}
	\GG_{j\al}(h)(x):=
\GG\big(M_j(\tfrac{\al\cdot}T)h(\al\cdot)\big)(x).
\end{equation}
 
For $j\in\{0,1,2,3\}$, introduce $f_1^{(j)}(t):=
\big(\frac{\dd}{\dd t}\big)^j f_1(t)$ and $\fnj j(t):=f_1^{(j)}(t/\sqrt n)$, where $f_1$ denotes the c.f.\ of $X_1$. 
Similarly, starting with $g_1:=g$ in place of $f_1$, define $\gnj j$, and then let 
$\dnj j:=\fnj j-\gnj j$ and $\hnj j:=\big|\fnj j\big|\vee\big|\gnj j\big|$; omit superscripts ${}^{(0)}$ and ${}^{[0]}$. 
Note that $f=\fn^n$ and hence $\sqrt{n}f'''=f_{31}+f_{32}+f_{33}$, where 
$f_{31}:=(n-1)(n-2)\fn^{n-3}\big(\fnj1\big)^3$, $f_{32}:=3(n-1)\fn^{n-2}\fnj1\fnj2$, and $f_{33}:=\fn^{n-1}\fnj3$; do similarly with $g$ and $g_1$ in place of $f$ and $f_1$. 
By Proposition~\ref{prop:quasi-cf}, $M_j(\tfrac{\al\cdot}T)f_{33}/\be_3$ is a quasi-c.f.\ and hence, by Proposition~\ref{prop:b-var}, $\GG_{j\al}(f_{33})\O\be_3$, for $j\in\{1,2\}$. 

So, it suffices to show that $\GG_{1\al}(f_{3k}-g_{3k})\O\be_3$ and $
\GG_{2\al}(f_{3k})\O\be_3$ for $k\in\{1,2
\}$.
%
This can be done in a straightforward manner using the following estimates for $j\in\{0,1,2,3\}$ and $|t|\le T$:\quad 
$M_1\O1$, $
M_2(\frac t{
T})\O
\frac{|t|}{
T}\O|t|\be_3/\sqrt n$, $\hn(t)^{n-j}\le e^{-ct^2}$ (where $c$ is a positive real number depending only on the choice of $c_T$), 
$\hnj1(t)\O|t|
/\sqrt n$, $\hnj2(t)\O1$, $|\dnj j(t)|\O\be_3(|t|/\sqrt n)^{3-j}$, and hence 
$\fn^{n-j}(t)-\gn^{n-j}(t)\O|t|^3 e^{-ct^2}\be_3/\sqrt n$; cf.\ e.g.\ \cite[Ch.\ V, Lemma~1]{pet75}. 
For instance, $|f_{31}-g_{31}|\O n^2(D_{311}+D_{312})$, 
where 
$D_{311}(t):=
\big(|\fn^{n-3}-\gn^{n-3}|\big(\hnj1\big)^3\big)(t)\O|t|^3 e^{-ct^2}\frac{\be_3}{\sqrt n}\big(\frac{|t|}{\sqrt n}\big)^3$ and 
$D_{312}(t):=\big(\hn^{n-3}\big(\hnj1\big)^2|\dnj1|\big)(t) 
\O e^{-ct^2}\,\big(\frac{|t|}{\sqrt n}\big)^2\,\be_3\big(\frac{|t|}{\sqrt n}\big)^2$, so that  
$\GG_{1\al}(f_{31}-g_{31})
\O 
\int_{-\infty}^\infty(t^6+t^4)e^{-ct^2}\be_3\,\frac{\dd t}{|t|}
\O\be_3$.  
\qed

Of course, the above argument is rather crude and yet it demonstrates that the method 
based on the smoothing inequalities \eqref{eq:ineq_k} is quite effective. 
It also strongly suggests that this method 
can be used further, in order to obtain an explicit and appropriately small upper bound on the constant factor $c_\nu$. 

Let us now discuss some of the refinements that could be used within the general framework of the above \emph{quick proof of \eqref{eq:BE nonunif}}. 

There, in particular, we 
needed to bound 
\begin{gather}
L(H):=\int_0^1 [\GG\big(H(\al\cdot)\big)(x)-\GG\big(H(\cdot)\big)
(x)]\,3\al^2\dd\al,
\end{gather}
where $H$ is of the form $M_1(\tfrac{\#}T)(f_{3k}-g_{3k})$ or $M_2(\tfrac{\#}T)f_{3k}$ for $k\in\{1,2\}$ -- recall \eqref{eq:ineq_k}, \eqref{eq:r}, and \eqref{eq:La_k}. 
Tacitly, that bounding was then done using the trivial inequalities  
\begin{equation}\label{eq:extra 2}
|L(H)|
\le2\sup_{\al\in(0,1]}|\GG\big(H(\al\cdot)\big)(x)|
\le2\, \frac1{2\pi}\,\int_\infty^\infty 
|H(t)|\frac{\dd t}{|t|},     	
\end{equation}
where in turn we used the definition \eqref{eq:GG} of $\GG$ and the trivial identity $|e^{-itx}|=1$ for real $t$ and $x$; 
the integral in \eqref{eq:extra 2} exists even in the Lebesgue sense, since $|f_{3k}(t)-g_{3k}(t)|=O(|t|)$ and $|M_2(t)|=O(|t|)$. 

In fact, the factor $2$ in the last bound in \eqref{eq:extra 2} on $|L(H)|$ can be removed, so that one have 
\begin{equation}\label{eq:w/out 2}
	|L(H)|\le\frac1{2\pi}\,\int_\infty^\infty |H(t)|\frac{\dd t}{|t|}. 
\end{equation}
Indeed, first of all note here that the factor $\al$ can be easily moved, in a way, from the argument of the general and hard to control function $H$ into that of the much simpler and more specific exponential function, using the simple identity 
\begin{equation}
	\GG\big(H(\al\cdot)\big)(x)=\GG\big(H(\cdot)\big)(\tfrac x\al), 
\end{equation}
which implies that 
\begin{align}
L(H)
&=\frac i{2\pi}\,\int_\infty^\infty I(tx)H(t)\frac{\dd t}t, \quad\text{where} \label{eq:L(H)=}\\ 	
I(u)&:=\int_0^1 (e^{-i u/\al}-e^{-i u})\,3\al^2\dd\al. 
\end{align}
%
%
Now \eqref{eq:w/out 2} follows immediately from 


\begin{proposition}\label{prop:w/out 2}
The expression 
$g( u):=|I(u)|^2$ 
is even in $ u\in\R$ and (strictly) increases from $0$ to $1$ as $| u|$ increases from $0$ to $\infty$; in particular, it follows that $|I(u)|\in[0,1)$ for all real $u$. 
Moreover, the function $g$ has the following generalized concavity property: 
$-u^3\big(u^{-5}g'(u)\big)'$ is completely monotone in $u>0$ (in Bernstein's sense -- see e.g.\ \cite[Chapter~2]{phelps}); in particular, $g(v^{1/6})$ is concave in $v>0$.  
\end{proposition}

Thus, the conclusion in Proposition~\ref{prop:w/out 2} that $|I(u)|\in[0,1)$ for all real $u$ can be seen as a rather sophisticated replacement for the trivial identity $|e^{-itx}|=1$ for real $t$ and $x$, which latter was used to obtain the rightmost bound in \eqref{eq:extra 2}. 

Moreover, one can easily obtain (and then use in \eqref{eq:L(H)=}) an upper bound on $|I(u)|$ which is significantly less than $1$ for small enough values of $|u|$. This can be done by closely bounding the values of $|I(u)|$ for a finite number of values of $u$ and then using the monotonicity property of $|I|$ provided by  Proposition~\ref{prop:w/out 2}. 

Graphs of $\Re I$, $\Im I$, and $|I|$ over the interval $[-6\pi,6\pi]$ are shown in Figure~\ref{fig:Untitled-1}. 
It seems plausible that $g(u)$ is concave in $u>0$; however, that probably would be hard to prove. 

\begin{figure}
[htbp]
	\centering
		\includegraphics[width=1.00\textwidth]{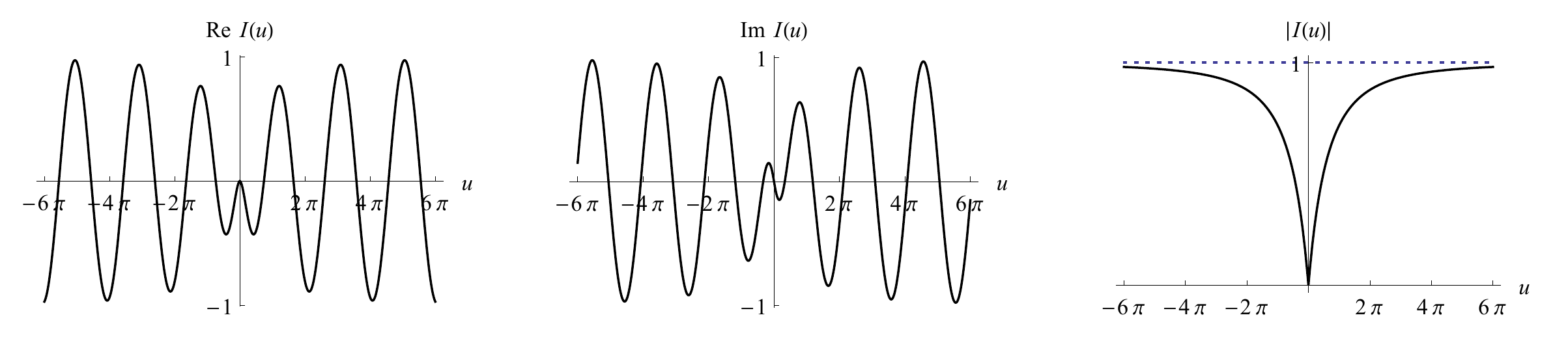}
	\caption{Graphs of $\Re I$, $\Im I$, and $|I|$. }
	\label{fig:Untitled-1}
\end{figure}

\begin{proof}[Proof of Proposition~\ref{prop:w/out 2}]  
Note that $I(-u)=\ol{I(u)}$ for all real $u$.  
So, the function $g=I\ol I$ is indeed even. 

Take now any real $u>0$. 
Integrating by parts and then changing the integration variable, one has 
\begin{equation}
	iu^{-3}I(u)=u^{-2}\int_0^1 e^{-i u/\al}\,\al\dd\al
	=\EE(u):=\EE_3(u),\quad\text{where}\quad\EE_j(u):=\int_u^\infty e^{-iz}\frac{\dd z}{z^j} 
\end{equation}
for $j>0$. 
So, $g(u)=u^6\,\EE(u)\ol{\EE(u)}$, 
\begin{align*}
	g_1(u)&:=\tfrac12\,u^{-5}g'(u)=3\EE(u)\ol{\EE(u)}-u^{-2}\Re\big( e^{iu}\EE(u)\big), \\
	u^3 g_1'(u)&=\frac1{u^2}-\Re\big((4+iu)e^{iu}\EE(u)\big) \\ 
	&=\frac1{u^2}-\Re\Big((4+iu)\int_u^\infty e^{-i(z-u)}\frac{\dd z}{z^3}\Big) \\ 
	&=\frac1{u^2}-\Re\Big((4+iu)\int_0^\infty e^{-iv}\frac{\dd v}{(v+u)^3}\Big) \\ 
	&=\int_0^\infty e^{-u s}s\dd s-\Re\Big((4+iu)\int_0^\infty e^{-iv}\dd v \int_0^\infty \tfrac12 e^{-(v+u)s}s^2\dd s\Big).  
\end{align*}
Noting now that $\int_0^\infty e^{-iv}e^{-v s}\dd v=\frac1{s+i}$ and $\Re\frac{4+iu}{s+i}=\frac{4s+u}{s^2+1}$ for all real $s>0$, and introducing $w_1(s):=2\frac{s-s^3}{s^2+1}$ and $w_2(s):=\frac{-s^2}{s^2+1}$, write
\begin{align*}
2u^3 g_1'(u)&=\int_0^\infty e^{-u s}w_1(s)\dd s+\int_0^\infty e^{-u s}w_2(s)\,u\dd s \\ 
&=\int_0^\infty e^{-u s}w_1(s)\dd s+\int_0^\infty e^{-u s}w_2'(s)\dd s 
=-2\int_0^\infty e^{-u s}\frac{s^5\dd s}{(s^2+1)^2},    
\end{align*}
which verifies the last sentence of the statement of Proposition~\ref{prop:w/out 2}; the second equality in the above display was obtained by taking the integral $\int_0^\infty e^{-u s}w_2(s)\,u\dd s$ by parts.  
Moreover, it follows that $g_1(u)$ is decreasing in $u>0$. 
At that, $g_1(\infty-)=0$, since $\EE(\infty-)=0$.  
So, on $(0,\infty)$ one has the following: $g_1>0$ and hence $g'>0$, and therefore $g$ is increasing. 
Since $g$ is even and obviously continuous, it follows that indeed $g(u)$ increases in $|u|$. 
Clearly, $g(0)=|I(0)|^2=0$. 
It remains only to show that $g(\infty-)=1$. 
Toward that end, integrate by parts to obtain the recursive relation $\EE_j(u)=-ie^{-iu}u^{-j}+ij\EE_{j+1}(u)$ for all $j>0$. In particular, it follows that $\EE_3(u)=-ie^{-iu}u^{-3}+3i\EE_4(u)$ and $\EE_4(u)\O u^{-4}+|\EE_5(u)|\le u^{-4}+\int_u^\infty\frac{\dd z}{z^5}\O u^{-4}$. 
Thus, $\EE(u)=\EE_3(u)=-ie^{-iu}u^{-3}+O(u^{-4})=-iu^{-3}\big(e^{-iu}+o(1)\big)$ and $g(u)=u^6\,|\EE(u)|^2\to1$ as $u\to\infty$. 
\end{proof}



One will also could use a better upper bound on $|f(t)|$ for a given real value of $t$, where $f$ is the c.f.\  of a r.v.\ $X$, say with $\E X=0$, $\E X^2=1$, and a given value of $\rho:=\E|X|^3$. Since 
\begin{equation}\label{eq:optimiz}
	|f(t)|=\sqrt{\E^2\cos tX+\E^2\sin tX}
	=\sup_{\th\in[0,2\pi]}(\cos\th\,\E\cos tX+\sin\th\,\E\sin tX)
	=\sup_{\th\in[0,2\pi]}\E\cos(tX-\th), 	
\end{equation}
the best upper bound on $|f(t)|$ under the given conditions is 
\begin{align}
	\mathcal{S}(t,\rho)
	&:=\sup\{|f(t)|\colon\E X=0, \E X^2=1, \E|X|^3=\rho\} \label{eq:S(t,rho)} \\ 
	&=\sup\{\E\cos(tX-\th)\colon\E X=0, \E X^2=1, \E|X|^3=\rho, \card\supp X\le4, \th\in[0,2\pi]\} \notag
	\\ 
	&=\textstyle{
	\sup\Big\{\sum\limits_1^4 p_j\cos(tx_j-\th)\colon\sum\limits_1^4 p_j=1, \sum\limits_1^4 p_j x_j=0, \sum\limits_1^4 p_j x_j^2=1, \sum\limits_1^4 p_j|x_j|^3=\rho
	},  \notag\\
	&\qquad\qquad\qquad\qquad\qquad\qquad\qquad\qquad\qquad\qquad\qquad\quad p_1,\dots,p_4\ge0,\ \th\in[0,2\pi]\Big\}; \notag  
\end{align}
$\card\supp$ denotes the cardinality of the support of (the distribution of) $X$; for the second equality here, one can use the known results by Hoeffding \cite{hoeff-extr} or Karr \cite{karr} or, somewhat more conveniently, Winkler \cite{winkler88} or Pinelis \cite[Propositions~5 and 6(v)]{ext}. 
Thus the optimization problem reduces to one in 9 variables: $p_1,\dots,p_4,x_1,\dots,x_4,\th$; in fact, one can easily solve the linear (or, more precisely, affine) restrictions $\sum_1^4 p_j=1, \sum_1^4 p_j x_j=0, \sum_1^4 p_j x_j^2=1, \sum_1^4 p_j|x_j|^3=\rho$ for $p_1,\dots,p_4$, and then only 5 variables will remain: $x_1,\dots,x_4,\th$, with the additional restrictions on $x_1,\dots,x_4$ to provide for the conditions $p_1,\dots,p_4\ge0$. 
For any given pair of values of $(t,\rho)$, it will not be overly hard to find a close upper bound on the supremum $\mathcal{S}(t,\rho)$. A difficulty here is that one has to deal with two parameters, $t$ and $\rho$, and obtain a close majorant of $\mathcal{S}(t,\rho)$ with, at least, discoverable and tractable patterns of monotonicity/convexity in $t$ and $\rho$, if not with a more or less explicit expression. 
Apparently the main difficulty in dealing with $\mathcal{S}(t,\rho)$ will be that the target function $\cos(t\cdot-\th)$ oscillates, whereas the function $(\cdot-w)_+^3$ in \cite[Lemma~3.4]{pin-hoeff-AIHP} is monotonic. 

Similar methods can be used to find a good upper bound on $|f(t)-g(t)|$, 
where $g=e^{-\cdot^2/2}$, the c.f.\  of the standard normal r.v.\ $Z$; in particular, one can start here by writing  
\begin{equation*}
	|f(t)-g(t)|
	=\sup_{\th\in[0,2\pi]}\big(\E\cos(tX-\th)-\E\cos(tZ-\th)\big)	=\sup_{\th\in[0,2\pi]}\big(\E\cos(tX-\th)-g(t)\cos\th\big)  	
\end{equation*}
in place of \eqref{eq:optimiz}. 
At this point, one also has an option to use Stein's method to bound $\E\cos(tX-\th)-\E\cos(tZ-\th)$.

\section{Constructions of the smoothing filter $M$
}\label{smoothing}
The following proposition was somewhat implicit in the paper \cite{bohman} by Bohman. 
\begin{proposition}\label{prop:bohman}
Let $p$ be any symmetric probability density function (p.d.f.) such that the function $\cdot p(\cdot)$ is integrable on $\R$. 
Take any real 
\begin{equation}\label{eq:ka_* Bo}
	\ka\ge\ka_*:=\frac1{\int_\R|x|p(x)\dd x}. 
\end{equation}
Let $\hat p$ stand, as usual, for the Fourier transform of $p$ \big(so that $\hat p(\cdot)=\int_\R e^{i x \cdot}p(x)\dd x$\big), and let then $\hat p'$ denote the derivative of $\hat p$ \big(which exists, since $\int_\R|x|p(x)\dd x<\infty$\big). 
Then the function 
\begin{equation}\label{eq:Mmy}
	M:=\hat p+i\ka\hat p' 
\end{equation}
is such that inequalities \eqref{eq:praw} hold 
for all r.v.\ $X$, all real $T>0$, and all real $x$. 
\end{proposition}

Because of the symmetry of $p$, in the conditions of Proposition~\ref{prop:bohman} the function $\hat p=\Re M$ is even and $\ka\hat p'=\Im M$ is odd, so that conditions \eqref{eq:M1,M2} hold. 
In order to satisfy the conditions \eqref{eq:M=0} and $M\in C^3$ as well, one may choose the symmetric p.d.f.\ $p_{0,2}$ defined by the formula 
\begin{equation}
	p_{0,2}(x):=
	\frac{32 \pi ^3}3\, \frac{1-\cos x}{x^2 \left(x^2-4 \pi ^2\right)^2}
\end{equation}
for real $x\notin\{-2\pi,0,2\pi\}$ and then let $M$ be as in \eqref{eq:Mmy} with $p=p_{0,2}$: 
\begin{equation}\label{eq:Mmy special}
	M=M_{0,2}:=\widehat{p_{0,2}}+i\ka\widehat{p_{0,2}}\,',  
\end{equation}
with any 
\begin{equation}
	\ka\ge\ka_{0,2}:=\frac1{\int_\R|x|p_{0,2}(x)\dd x}=0.3418\dots. 
\end{equation}
Then $M_{0,2}\in C^3$, since $M'''=\hat p'''+i\ka\hat p''''$ and 
\begin{equation}\label{eq:pMy-4th}
	\int_\R x^4\, p_{0,2}(x)\dd x<\infty. 
\end{equation}
Moreover, it is clear that $p_{0,2}$ is the restriction to $\R$ of an entire analytic function of exponential type $1$; so, by 
the Paley--Wiener theory (see e.g.\ \cite[Section~43]{donoghue}), the condition \eqref{eq:M=0} holds as well, with $M_{0,2}$ in place of $M$. In fact, 
\begin{align}
	\Re M_{0,2}(t)&=\widehat{p_{0,2}}(t)=\left(
	\frac{2+\cos 2 \pi t}3\,(1-|t|)
	+\frac{\sin 2 \pi  |t|}{2 \pi } 
	\right) \ii\{|t|<1\}
	\quad\text{and} \\
\Im M_{0,2}(t)&=\ka\widehat{p_{0,2}}\,'(t)=
-\ka\,\frac{\sign t}{3}\, \big[2 \pi(1-|t|) \sin 2 \pi|t| + 4 \sin ^2(\pi t)\big] \ii\{|t|<1\}  	 
\end{align}
for all real $t$. 
Graphs of $p_{0,2}$, $\Re M_{0,2}$, and $\Im M_{0,2}$ with $\ka=\ka_{0,2}$ are shown in Fig.\ \ref{fig:KMy}.
\begin{figure}[h]	
\centering	\includegraphics[width=1.00\textwidth]{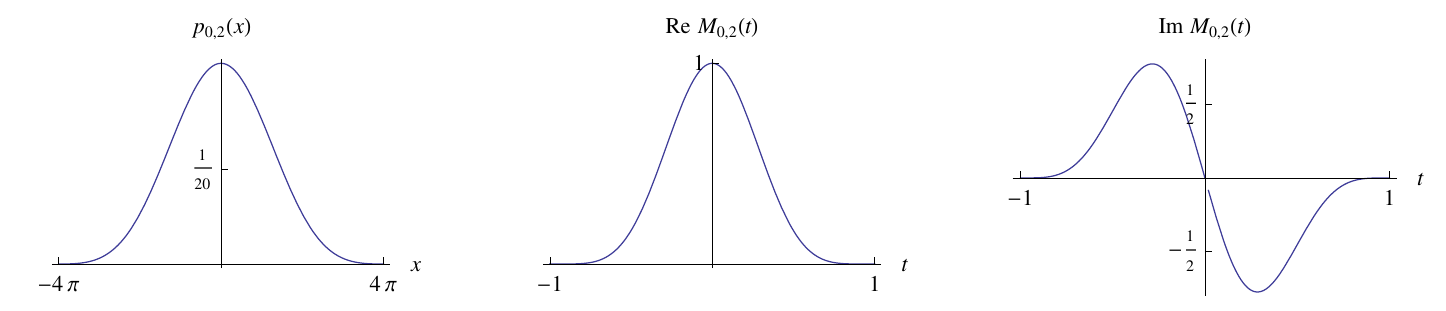}
	\caption{Graphs of $p_{0,2}$, $\Re M_{0,2}$, and $\Im M_{0,2}$ with $\ka=\ka_{0,2}$. }
	\label{fig:KMy}
\end{figure}

More generally, in order that a function $M$ as in \eqref{eq:Mmy} satisfy the conditions \eqref{eq:M=0} and $M\in C^3$, it is enough that $\hat p$ be smooth enough and such that \eqref{eq:M=0} holds with $\hat p$ in place of $M$. Therefore, the following well-known characterization 
is useful. 


\newcommand{\upmin}[1]{{}^{\raisebox{-1pt}{$\scriptstyle{#1}$}}}

\begin{proposition}\label{prop:c.f.} (See e.g.\ \cite[Theorem~4.2.4]{lukacs}.) 
A function $f\colon\R\to\C$ is the c.f.\ of an absolutely continuous distribution on $\R$ if and only if $f(0)=1$ and $f=g*\ol{g}\upmin-$ 
for some (possibly complex-valued) function $g\in L^2(\R)$. 
Here and in the sequel, as usual, the symbol $*$ stands for the convolution,  
the bar denotes the complex conjugation, $g^-(\cdot):=g(-\cdot)$, and $\ol{g}\upmin-:=(\ol{g})\upmin-=\ol{g\upmin-}$.  
\end{proposition}

Indeed, take any smooth enough nonzero function $g\colon\R\to\C$ such that 
$g(t)=0$ for all real $t\notin[a,b]$, where $a$ and $b$ are any real numbers such that $a<b$. 
Then,
by Proposition~\ref{prop:c.f.}, the function $f:=g*\ol{g}\upmin-/\|g\|_2^2$ is the c.f.\ of an absolutely continuous distribution on $\R$, 
$f$ is smooth enough, and $f(t)=0$ for all real $t\notin[-T,T]$, where $T:=b-a$. 
At that, if $g$ is real-valued, then $f$ is even. 
To spell-out the ``smooth enough'' condition and conclusion here, one can easily check that, if $g\in C^j$ and $h\in C^k$ for some $j$ and $k$ in $\N$ and (say) $|g(t)|+|h(t)|=0$  for some real $T>0$ and all real $t\notin[0,T]$, then $g*h\in C^{j+k}$, with $(g*h)^{(j+k)}=g^{(j)}*h^{(k)}$. 


One can use Proposition~\ref{prop:bohman} and Proposition~\ref{prop:c.f.} to optimize properties of the filter $M$ -- say by taking $g$ to be an arbitrary nonzero real-valued spline 
of a high enough order and/or with a large enough subintervals of the interval $[0,T]$, extending $g$ to $\R$ by letting $g(t):=0$ for all real $t\notin[0,T]$, letting then $f:=g*\ol{g}\upmin-/\|g\|_2^2$, defining $M$ as in \eqref{eq:Mmy}, and finally (quasi-)optimizing with respect to the parameters of the spline.

While the construction described in Proposition~\ref{prop:bohman} is comparatively simple, it appears somewhat too rigid and wasteful. Indeed, in order that the imaginary part $M_2=\ka\hat p'$ of the function $M$ in \eqref{eq:Mmy} be thrice differentiable \big(as needed or almost needed in the \emph{quick proof} beginning on page~\pageref{quick-proof}\big), the real part $M_1=\hat p$ of $M$ must be four times differentiable; equivalently \big(cf.\ \eqref{eq:pMy-4th}\big), the density $p$ must have light enough tails so that $\int_\R x^4\, p(x)\dd x<\infty$. Together with the filtering condition \eqref{eq:M=0}, the condition of extra smoothness of $M_1=\hat p$/extra lightness of the tails of $p$ may result in a smoothing filter $M$ which is not as good as it can be, thus compromising the quality of the approximation by the upper and lower bounds in \eqref{eq:ineq_k}. 

A more flexible and potentially better construction of the smoothing filter $M$ can be given as follows. 
As in Proposition~\ref{prop:bohman}, let us start with an arbitrary symmetric p.d.f.\ $p$, whose Fourier transform $\hat p$ is intended to be $M_1=\Re M$. Accordingly, let us assume right away that 
\begin{equation}\label{eq:hat p=0}
	\hat p(t)=0\quad\text{if}\quad|t|>1; 
\end{equation}
cf.\ \eqref{eq:M=0}. 
Note that the smoothing filter  
as in \eqref{eq:Mmy} is the Fourier transform of the function 
\begin{equation}\label{eq:tpBo}
x\mapsto p(x)(1-\ka x), 	
\end{equation}
which differs relatively much from the ``original'' p.d.f.\ $x\mapsto p(x)$ when $|x|$ is large. 
To address this concern, let us replace the ``large'' factor $x$ in \eqref{eq:tpBo} by a ``tempered'' and, essentially, more general factor $G(x)$ such that $G\colon\R\to\R$ is a strictly increasing odd function of bounded variation, whose Fourier--Stieltjes transform $\widehat{\dd G}(\cdot)=\int_\R e^{ix\cdot}\dd G(x)$ satisfies the condition 
\begin{equation}\label{eq:hat dG=0}
	\widehat{\dd G}(t)=0\quad\text{if}\quad|t|>\ga,  
\end{equation}
for some real $\ga>0$.   
The no-high-frequency-component condition \eqref{eq:hat dG=0} implies, by the mentioned 
Paley--Wiener theory, that the function $G$ is the restriction to $\R$ of an entire analytic function of exponential type $\ga$ and, in particular, is infinitely many times differentiable. 
Without loss of generality, assume that the function $\frac12+G$ is a d.f.  

As mentioned above, instead of the ``harsh'' tilting \eqref{eq:tpBo} of the p.d.f.\ $p$, we consider the ``tempered'' tilting: 
\begin{equation}\label{eq:tp}
x\mapsto \tp(x):=p(x)\big(1-\ka G(x)\big),  	
\end{equation}
for any real 
\begin{equation}\label{eq:ka_* temp}
	\ka\ge\ka_*:=\frac1{2\int_0^\infty p(x)G(x)\dd x}. 
\end{equation}
%
Note that $0\le G(x)<\frac12=G(\infty-)$ for all real $x>0$; also, as discussed previously, the condition \eqref{eq:hat p=0} implies that $p$ is the restriction to $\R$ of an entire analytic function, and so, $p>0$ almost everywhere on $\R$. 
Therefore and by the symmetry of $p$, one has $0<\int_0^\infty p(x)G(x)\dd x<\frac12\,\int_0^\infty p(x)\dd x=\frac14$ and hence $\ka_*>2$ and $\ka>2$. 
It follows that there exists a unique root $x_\ka\in(0,\infty)$ of the equation  
\begin{equation}
	1-\ka G(x_\ka)=0. 
\end{equation}
Hence, $\tp\ge0$ on the interval $(-\infty,x_\ka]$ and $\tp\le0$ on $[x_\ka,\infty)$, so that the function 
\begin{equation}
	\tF(\cdot):=\int_{-\infty}^\cdot\tp(y)\dd y
\end{equation}
is nondecreasing on $(-\infty,x_\ka]$ and nonincreasing on $[x_\ka,\infty)$. 
At that, $\tF(\infty-)=1$, since $p$ is an even p.d.f.\ and the bounded function $G$ is odd; also, clearly 
$\tF(-\infty+)=0$. 
Moreover, $\tF(0)=\frac12-\ka\int_{-\infty}^0 p(x)G(x)\dd x=\frac12+\ka\int_0^\infty p(x)G(x)\dd x 
=\frac12+\frac\ka{2\ka_*}\ge1$. It follows that 
\begin{equation}\label{eq:ii<}
	\ii\{y\ge0\}\le\tF(y)\quad\text{for all real $y$.}  
\end{equation}
Let now $X$ be any r.v.\ and let $f$ by its c.f.:
\begin{equation}
	f(\cdot):=\E e^{iX\cdot}. 
\end{equation}
Then $\P(X\le x)=\E\ii\{x-X\ge0\}\le\E\tF(x-X)$, by \eqref{eq:ii<}; that is, 
\begin{equation}\label{eq:le}
	\P(X\le x)\le\int_\R\tF(x-y)\P(X\in\dd y), 
\end{equation}
for all real $x$.  
Define now $M$ as the Fourier transform of $\tp$, so that  
\begin{equation}\label{eq:M1,M2 temp}
M=\hat\tp,\quad
	M_1=\Re M=\hat p,\quad\text{and}\quad 
	M_2=\Im M=i\ka\, \widehat{pG},  
\end{equation}
by \eqref{eq:tp}. 
Note that the Fourier--Stieltjes transform of the function $\int_\R\tF(\cdot-y)\P(X\in\dd y)$ is the Fourier transform of $\int_\R\tp(\cdot-y)\P(X\in\dd y)$, which in turn is $\hat\tp f=Mf
$. 
Then, in view of Proposition~\ref{prop:b-var}, \eqref{eq:le} means that the last inequality in \eqref{eq:praw} holds for $T=1$; that it holds for any real $T>0$ now follows by simple re-scaling, since \eqref{eq:ii<} obviously implies $\ii\{y\ge0\}\le\tF(Ty)$
for all real $y$ and all $T>0$. 
Similarly or using the reflection $x\mapsto-x$, one can see that the first inequality in \eqref{eq:praw} holds as well.


To compute $M_2$ in \eqref{eq:M1,M2 temp}, we need to express $\widehat{pG}$ in terms $\hat p$ and $\widehat{\dd G}$. To simplify the derivation, assume the condition $\hat p\in C^1$, as well as the previously stated conditions \eqref{eq:hat p=0} and \eqref{eq:hat dG=0}; these conditions will hold in the applications anyway. 
Then one can see that for all real $u$ 
\begin{equation}\label{eq:hat pG}
	\widehat{pG}(u)=
	\frac i{2\pi}\pv\int_{-\infty}^\infty \hat p(u-s)\widehat{\dd G}(s)\frac{\dd s}s
	=\frac i{2\pi}\int_\R \frac{\hat p(u-s)-\hat p(u)}s\,\widehat{\dd G}(s)\dd s; 
\end{equation}
the latter equality here holds because the function $G$ was assumed odd, and hence the function $\widehat{\dd G}$ is even; 
the latter integral in \eqref{eq:hat pG} may be understood in the Lebesgue sense, in view of the conditions $\hat p\in C^1$ and \eqref{eq:hat dG=0}. 
It follows from \eqref{eq:hat pG}, \eqref{eq:hat p=0}, and \eqref{eq:hat dG=0} 
that 
\begin{equation}\label{eq:hat pG=0}
	\widehat{pG}(t)=0\quad\text{if}\quad|t|>1+\ga.   
\end{equation}
To verify the first equality in \eqref{eq:hat pG}, one can write 
%
\begin{align*}
	p(x)G(x)&=\Big(\frac1{2\pi}\int_\R e^{-itx}\hat p(t)\dd t\Big)
	\, \Big(\frac i{2\pi}\pv\int_{-\infty}^\infty e^{-isx}\widehat{\dd G}(s)\frac{\dd s}s\Big) \\ 
	&=\frac1{2\pi}\int_\R e^{-iux}\dd u
	\, \frac i{2\pi}\pv\int_{-\infty}^\infty \hat p(u-s)\widehat{\dd G}(s)\frac{\dd s}s;
\end{align*}
the second equality here is justified because of the second equality in \eqref{eq:hat pG} and the inequality $\big|\int_\R \frac{\hat p(u-s)-\hat p(u)}s\,\widehat{\dd G}(s)\dd s\big|
\le2\ga\ii\{|u|\le1+\ga\}\max_{|t|\le1}|\hat p'(t)|$ for all real $u$. 

Note that the first integral in \eqref{eq:hat pG} is a convolution. 
One can also integrate by parts to represent $\widehat{pG}$ as a convolution-smoothing of the derivative $\hat p'$: 
\begin{equation}
	\widehat{pG}=\tfrac i{2\pi}\, \hat p' * \widetilde{\widehat{\dd G}}, 
\end{equation}
where
\begin{equation}\label{eq:tilhat dG}
	\widetilde{\widehat{\dd G}}(t):=\pv\int_{-\infty}^t \widehat{\dd G}(s)\frac{\dd s}s
	:=\lim_{\vp\downarrow0}\int_{(-\infty,t)\setminus(-\vp,\vp)}\widehat{\dd G}(s)\frac{\dd s}s 
	=\int_{-\ga}^{-|t|} \widehat{\dd G}(s)\frac{\dd s}s\,\ii\{|t|<\ga\}
\end{equation}
for all real $t$. 
This follows because for all real $u$ 
\begin{align*}
\int_\R \frac{\hat p(u-s)-\hat p(u)}s\,\widehat{\dd G}(s)\dd s
&=\int_\R \widehat{\dd G}(s)\frac{\dd s}s 
\Big(\int_u^{u-s}\hat p'(v)\dd v\,\ii\{s<0\}-\int_{u-s}^u \hat p'(v)\dd v\,\ii\{s>0\}\Big)\\ 
&=\int_u^\infty \hat p'(v)\dd v \int_{-\infty}^{u-v}\widehat{\dd G}(s)\frac{\dd s}s 
- \int_{-\infty}^u \hat p'(v)\dd v \int_{u-v}^\infty\widehat{\dd G}(s)\frac{\dd s}s \\
&=\int_\R \hat p'(v)\dd v \widetilde{\widehat{\dd G}}(u-v),   
\end{align*}
since the function $\widehat{\dd G}$ is even. 
The latter condition or \eqref{eq:tilhat dG} also shows that the function $\widetilde{\widehat{\dd G}}$ is even.  
Moreover, since $\widehat{\dd G}(s)\to1$ as $s\to0$, \eqref{eq:tilhat dG} yields 
\begin{equation}
	\widetilde{\widehat{\dd G}}(t)\sim\ln|t| 
\end{equation}
as $t\to0$; thus, the function $\widetilde{\widehat{\dd G}}$ is mildly singular in a neighborhood of $0$. 
For instance, $\widetilde{\widehat{\dd G}}(t)\equiv(1-|t|+\ln|t|)(\ii\{|t|<1\}$ if 
$\widehat{\dd G}(t)\equiv(1-|t|)(\ii\{|t|<1\}$. 

Note also that in the case \big(prevented by the condition \eqref{eq:hat dG=0}\big) when $\widehat{\dd G}=1$ on $\R$, the function 
$-2i\,\widehat{pG}$ would be the Hilbert transform of the function $\hat p$; see e.g.\ \cite[Chapter~V]{titch_Fourier}. 

It follows from \eqref{eq:M1,M2 temp} and \eqref{eq:hat pG=0} that 
\begin{equation}\label{eq:M2=0,ga}
	M_2(t)=0\quad\text{if}\quad|t|>1+\ga.     
\end{equation}
This condition on $M_2$ is obviously weaker than the condition 
\begin{equation}\label{eq:M2=0}
	M_2(t)=0\quad\text{if}\quad|t|>1,      
\end{equation}
following from \eqref{eq:M=0} and \eqref{eq:M1,M2}. 
However, by \eqref{eq:M1,M2 temp} and \eqref{eq:hat p=0}, one still has $M_1(t)=0$ if $|t|>1$, whereas the condition \eqref{eq:M2=0} was used in the \emph{quick proof} beginning on page~\pageref{quick-proof} only to bound two terms, $\GG_{2\al}(f_{31})$ and $\GG_{2\al}(f_{32})$. 
Therefore, one may expect the adverse impact of the weakening of the condition \eqref{eq:M2=0} to \eqref{eq:M2=0,ga} to be rather limited and likely more than compensated for by the advantages provided by the more 
flexible construction of the smoothing filter $M$, with the tempered tilting of $M_1$.  
Moreover, the latter construction is, essentially, more general, Indeed, for instance,  
one may always include $G$ into the scale family $(G_\al)_{\al>0}:=\big(G(\frac\cdot\al)\big)_{\al>0}$, and then the tempered tilting \eqref{eq:tp} will be close to the harsh tilting \eqref{eq:tpBo} for large $\al>0$  provided that $G'(0)\ne0$. Indeed, $G_\al(x)\sim\frac{G'(0)}\al\,x$ for each real $x\ne0$ as $\al\to\infty$; of course, at that the value of $\ka=\ka_\al$ in \eqref{eq:tp} with $G=G_\al$ will be quite different from that in \eqref{eq:tpBo}; in fact, the value of $\ka_*$ in \eqref{eq:ka_* temp} with $G=G_\al$ will then be asymptotically equivalent to the value of $\ka_*$ in \eqref{eq:ka_* Bo} times $\frac\al{G'(0)}$, provided that $\int_\R|x|p(x)\dd x<\infty$. 
At that, the value of $\ga$ in \eqref{eq:hat dG=0} for $G$ will be replaced by the corresponding value $\ga_\al:=\frac\ga\al$ for $G_\al$, so that $\ga_\al\to0$ as $\al\to\infty$.

\bibliographystyle{abbrv}


\bibliography{C:/Users/Iosif/Dropbox/mtu/bib_files/citations12.13.12}
\end{document}